\documentclass{amsart}
\usepackage{amsmath}
\usepackage{amssymb}
\usepackage{amsbsy, amsthm,amstext, amsopn}
\usepackage[all]{xy}
\usepackage{amsfonts}
\usepackage{amscd}
\usepackage{stmaryrd}
\usepackage{color}
\usepackage[numbers, square, sort, comma]{natbib}

\usepackage{graphicx}

\newtheorem{thm}{Theorem} [section]
\newtheorem{lemma}[thm]{Lemma}

\theoremstyle{definition}

\newtheorem{example}[thm]{Example}

\newtheorem{ansatz}[thm]{Ansatz}

\newtheorem{remark}[thm]{Remark}

\begin{document}

\numberwithin{equation}{section}

\newcommand{\hs}{\mbox{\hspace{.4em}}}
\newcommand{\ds}{\displaystyle}
\newcommand{\bd}{\begin{displaymath}}
\newcommand{\ed}{\end{displaymath}}
\newcommand{\bcd}{\begin{CD}}
\newcommand{\ecd}{\end{CD}}

\newcommand{\on}{\operatorname}
\newcommand{\proj}{\operatorname{Proj}}
\newcommand{\bproj}{\underline{\operatorname{Proj}}}

\newcommand{\spec}{\operatorname{Spec}}
\newcommand{\Spec}{\operatorname{Spec}}
\newcommand{\bspec}{\underline{\operatorname{Spec}}}
\newcommand{\pline}{{\mathbf P} ^1}
\newcommand{\aline}{{\mathbf A} ^1}
\newcommand{\pplane}{{\mathbf P}^2}
\newcommand{\aplane}{{\mathbf A}^2}
\newcommand{\coker}{{\operatorname{coker}}}
\newcommand{\ldb}{[[}
\newcommand{\rdb}{]]}

\newcommand{\Sym}{\operatorname{Sym}^{\bullet}}
\newcommand{\Symp}{\operatorname{Sym}}
\newcommand{\Pic}{\bf{Pic}}
\newcommand{\Aut}{\operatorname{Aut}}
\newcommand{\PAut}{\operatorname{PAut}}

\newcommand{\too}{\twoheadrightarrow}
\newcommand{\C}{{\mathbf C}}
\newcommand{\Z}{{\mathbf Z}}
\newcommand{\Q}{{\mathbf Q}}
\newcommand{\Cx}{{\mathbf C}^{\times}}
\newcommand{\Cbar}{\overline{\C}}
\newcommand{\Cxbar}{\overline{\Cx}}
\newcommand{\cA}{{\mathcal A}}
\newcommand{\cS}{{\mathcal S}}
\newcommand{\cV}{{\mathcal V}}
\newcommand{\cM}{{\mathcal M}}
\newcommand{\bA}{{\mathbf A}}
\newcommand{\cB}{{\mathcal B}}
\newcommand{\cC}{{\mathcal C}}
\newcommand{\cD}{{\mathcal D}}
\newcommand{\D}{{\mathcal D}}
\newcommand{\cs}{{\mathbf C} ^*}
\newcommand{\boldc}{{\mathbf C}}
\newcommand{\cE}{{\mathcal E}}
\newcommand{\cF}{{\mathcal F}}
\newcommand{\bF}{{\mathbf F}}
\newcommand{\cG}{{\mathcal G}}
\newcommand{\G}{{\mathbb G}}
\newcommand{\cH}{{\mathcal H}}
\newcommand{\CI}{{\mathcal I}}
\newcommand{\cJ}{{\mathcal J}}
\newcommand{\cK}{{\mathcal K}}
\newcommand{\cL}{{\mathcal L}}
\newcommand{\baL}{{\overline{\mathcal L}}}

\newcommand{\fF}{{\mathfrak F}}
\newcommand{\Mf}{{\mathfrak M}}
\newcommand{\bM}{{\mathbf M}}
\newcommand{\bm}{{\mathbf m}}
\newcommand{\cN}{{\mathcal N}}
\newcommand{\theo}{\mathcal{O}}
\newcommand{\cP}{{\mathcal P}}
\newcommand{\cR}{{\mathcal R}}
\newcommand{\Pp}{{\mathbb P}}
\newcommand{\boldp}{{\mathbf P}}
\newcommand{\boldq}{{\mathbf Q}}
\newcommand{\bbL}{{\mathbf L}}
\newcommand{\cQ}{{\mathcal Q}}
\newcommand{\cO}{{\mathcal O}}
\newcommand{\Oo}{{\mathcal O}}
\newcommand{\cY}{{\mathcal Y}}
\newcommand{\OX}{{\Oo_X}}
\newcommand{\OY}{{\Oo_Y}}
\newcommand{\otY}{{\underset{\OY}{\ot}}}
\newcommand{\otX}{{\underset{\OX}{\ot}}}
\newcommand{\cU}{{\mathcal U}}\newcommand{\cX}{{\mathcal X}}
\newcommand{\cW}{{\mathcal W}}
\newcommand{\boldz}{{\mathbf Z}}
\newcommand{\qgr}{\operatorname{q-gr}}
\newcommand{\gr}{\operatorname{gr}}
\newcommand{\rk}{\operatorname{rk}}
\newcommand{\SH}{{\underline{\operatorname{Sh}}}}
\newcommand{\End}{\operatorname{End}}
\newcommand{\uEnd}{\underline{\operatorname{End}}}
\newcommand{\Hom}{\operatorname{Hom}}
\newcommand{\uHom}{\underline{\operatorname{Hom}}}
\newcommand{\uHomY}{\uHom_{\OY}}
\newcommand{\uHomX}{\uHom_{\OX}}
\newcommand{\Ext}{\operatorname{Ext}}
\newcommand{\bExt}{\operatorname{\bf{Ext}}}
\newcommand{\Tor}{\operatorname{Tor}}

\newcommand{\inv}{^{-1}}
\newcommand{\airtilde}{\widetilde{\hspace{.5em}}}
\newcommand{\airhat}{\widehat{\hspace{.5em}}}
\newcommand{\nt}{^{\circ}}
\newcommand{\del}{\partial}

\newcommand{\supp}{\operatorname{supp}}
\newcommand{\GK}{\operatorname{GK-dim}}
\newcommand{\hd}{\operatorname{hd}}
\newcommand{\id}{\operatorname{id}}
\newcommand{\res}{\operatorname{res}}
\newcommand{\lrar}{\leadsto}
\newcommand{\im}{\operatorname{Im}}
\newcommand{\HH}{\operatorname{H}}
\newcommand{\TF}{\operatorname{TF}}
\newcommand{\Bun}{\operatorname{Bun}}

\newcommand{\F}{\mathcal{F}}
\newcommand{\Ff}{\mathbb{F}}
\newcommand{\nthord}{^{(n)}}
\newcommand{\Gr}{{\mathfrak{Gr}}}

\newcommand{\Fr}{\operatorname{Fr}}
\newcommand{\GL}{\operatorname{GL}}
\newcommand{\gl}{\mathfrak{gl}}
\newcommand{\SL}{\operatorname{SL}}
\newcommand{\ff}{\footnote}
\newcommand{\ot}{\otimes}
\def\Ext{\operatorname {Ext}}
\def\Hom{\operatorname {Hom}}
\def\Ind{\operatorname {Ind}}
\def\bbZ{{\mathbb Z}}

\newcommand{\nc}{\newcommand}
\nc{\ol}{\overline} \nc{\cont}{\on{cont}} \nc{\rmod}{\on{mod}}
\nc{\Mtil}{\widetilde{M}} \nc{\wb}{\overline} 
\nc{\wh}{\widehat}  \nc{\mc}{\mathcal}
\nc{\mbb}{\mathbb}  \nc{\K}{{\mc K}} \nc{\Kx}{{\mc K}^{\times}}
\nc{\Ox}{{\mc O}^{\times}} \nc{\unit}{{\bf \on{unit}}}
\nc{\boxt}{\boxtimes} \nc{\xarr}{\stackrel{\rightarrow}{x}}

\nc{\Ga}{\G_a}
 \nc{\PGL}{{\on{PGL}}}
 \nc{\PU}{{\on{PU}}}

\nc{\h}{{\mathfrak h}} \nc{\kk}{{\mathfrak k}}
 \nc{\Gm}{\G_m}
\nc{\Gabar}{\wb{\G}_a} \nc{\Gmbar}{\wb{\G}_m} \nc{\Gv}{G^\vee}
\nc{\Tv}{T^\vee} \nc{\Bv}{B^\vee} \nc{\g}{{\mathfrak g}}
\nc{\gv}{{\mathfrak g}^\vee} \nc{\BRGv}{\on{Rep}\Gv}
\nc{\BRTv}{\on{Rep}T^\vee}
 \nc{\Flv}{{\mathcal B}^\vee}
 \nc{\TFlv}{T^*\Flv}
 \nc{\Fl}{{\mathfrak Fl}}
\nc{\BRR}{{\mathcal R}} \nc{\Nv}{{\mathcal{N}}^\vee}
\nc{\St}{{\mathcal St}} \nc{\ST}{{\underline{\mathcal St}}}
\nc{\Hec}{{\bf{\mathcal H}}} \nc{\Hecblock}{{\bf{\mathcal
H_{\alpha,\beta}}}} \nc{\dualHec}{{\bf{\mathcal H^\vee}}}
\nc{\dualHecblock}{{\bf{\mathcal H^\vee_{\alpha,\beta}}}}
\newcommand{\ramBun}{{\bf{Bun}}}
\newcommand{\ramBuno}{\ramBun^{\circ}}

\nc{\Buntheta}{{\bf Bun}_{\theta}} \nc{\Bunthetao}{{\bf
Bun}_{\theta}^{\circ}} \nc{\BunGR}{{\bf Bun}_{G_\BR}}
\nc{\BunGRo}{{\bf Bun}_{G_\BR}^{\circ}}
\nc{\HC}{{\mathcal{HC}}}
\nc{\risom}{\stackrel{\sim}{\to}} \nc{\Hv}{{H^\vee}}
\nc{\bS}{{\mathbf S}}
\def\BRep{\operatorname {Rep}}
\def\Conn{\operatorname {Conn}}

\nc{\Vect}{{\operatorname{Vect}}}
\nc{\Hecke}{{\operatorname{Hecke}}}

\newcommand{\ZZ}{{Z_{\bullet}}}
\nc{\HZ}{{\mc H}\ZZ} \nc{\eps}{\epsilon}

\nc{\CN}{\mathcal N} \nc{\BA}{\mathbb A}

 \nc{\BB}{\mathbb B}

\nc{\ul}{\underline}

\nc{\bn}{\mathbf n} \nc{\Sets}{{\on{Sets}}} \nc{\Top}{{\on{Top}}}
\nc{\IntHom}{{\mathcal Hom}}

\nc{\Simp}{{\mathbf \Delta}} \nc{\Simpop}{{\mathbf\Delta^\circ}}

\nc{\Cyc}{{\mathbf \Lambda}} \nc{\Cycop}{{\mathbf\Lambda^\circ}}

\nc{\Mon}{{\mathbf \Lambda^{mon}}}
\nc{\Monop}{{(\mathbf\Lambda^{mon})\circ}}

\nc{\Aff}{{\on{Aff}}} \nc{\Sch}{{\on{Sch}}}

\nc{\bul}{\bullet}
\nc{\module}{{\operatorname{-mod}}}

\nc{\dstack}{{\mathcal D}}

\nc{\BL}{{\mathbb L}}

\nc{\BD}{{\mathbb D}}

\nc{\BR}{{\mathbb R}}

\nc{\BT}{{\mathbb T}}

\nc{\SCA}{{\mc{SCA}}}
\nc{\DGA}{{\mc DGA}}

\nc{\DSt}{{DSt}}

\nc{\lotimes}{{\otimes}^{\mathbf L}}

\nc{\bs}{\backslash}

\nc{\Lhat}{\widehat{\mc L}}

\newcommand{\Coh}{\on{Coh}}

\nc{\QCoh}{QC}
\nc{\QC}{QC}
\nc{\Perf}{\on{Perf}}
\nc{\Cat}{{\on{Cat}}}
\nc{\dgCat}{{\on{dgCat}}}
\nc{\bLa}{{\mathbf \Lambda}}

\nc{\BRHom}{\mathbf{R}\hspace{-0.15em}\on{Hom}}
\nc{\BREnd}{\mathbf{R}\hspace{-0.15em}\on{End}}
\nc{\colim}{\on{colim}}
\nc{\oo}{\infty}
\nc{\Mod}{\on{Mod} }

\nc\fh{\mathfrak h}
\nc\al{\alpha}
\nc\la{\alpha}
\nc\BGB{B\bs G/B}
\nc\QCb{QC^\flat}
\nc\qc{\on{QC}}

\def\w{\wedge}
\nc{\vareps}{\varepsilon}

\nc{\fg}{\mathfrak g}

\nc{\Map}{\on{Map}} \nc{\fX}{\mathfrak X}

\nc{\ch}{\check}
\nc{\fb}{\mathfrak b} \nc{\fu}{\mathfrak u} \nc{\st}{{st}}
\nc{\fU}{\mathfrak U}
\nc{\fZ}{\mathfrak Z}
\nc{\fB}{\mathfrak B}

 \nc\fc{\mathfrak c}
 \nc\fs{\mathfrak s}

\nc\fk{\mathfrak k} \nc\fp{\mathfrak p}
\nc\fq{\mathfrak q}

\nc{\BRP}{\mathbf{RP}} \nc{\rigid}{\text{rigid}}
\nc{\glob}{\text{glob}}

\nc{\cI}{\mathcal I}

\nc{\La}{\mathcal L}

\nc{\quot}{/\hspace{-.25em}/}

\nc\aff{\it{aff}}
\nc\BS{\mathbb S}

\nc\Loc{{\mc Loc}}
\nc\Tr{{\on{Tr}}}
\nc\Ch{{\mc Ch}}

\nc\ftr{{\mathfrak {tr}}}
\nc\fM{\mathfrak M}

\nc\Id{\operatorname{Id}}

\nc\bimod{\on{-bimod}}

\nc\ev{\operatorname{ev}}
\nc\coev{\operatorname{coev}}

\nc\pair{\operatorname{pair}}
\nc\kernel{\operatorname{kernel}}

\nc\Alg{\operatorname{Alg}}

\nc\init{\emptyset_{\text{\em init}}}
\nc\term{\emptyset_{\text{\em term}}}

\nc\Ev{\on{Ev}}
\nc\Coev{\on{Coev}}

\nc\es{\emptyset}
\nc\m{\text{\it min}}
\nc\M{\text{\it max}}
\nc\cross{\text{\it cr}}
\nc\tr{\on{tr}}

\nc\perf{\on{-perf}}
\nc\inthom{\mathcal Hom}
\nc\intend{\mathcal End}

\newcommand{\Sh}{\mathit{Sh}}

\nc{\Comod}{\on{Comod}}
\nc{\cZ}{\mathcal Z}

\def\interiorsymbol {\on{int}}

\nc\frakf{\mathfrak f}
\nc\fraki{\mathfrak i}
\nc\frakj{\mathfrak j}
\nc\BP{\mathbb P}
\nc\stab{st}
\nc\Stab{St}

\nc\fN{\mathfrak N}
\nc\fT{\mathfrak T}
\nc\fV{\mathfrak V}

\nc\Ob{\on{Ob}}

\nc\fC{\mathfrak C}
\nc\Fun{\on{Fun}}

\nc\Null{\on{Null}}

\nc\BC{\mathbb C}

\nc\loc{\on{Loc}}

\nc\hra{\hookrightarrow}
\nc\fL{\mathfrak L}
\nc\R{\mathbb R}
\nc\CE{\mathcal E}

\nc\sK{\mathsf K}
\nc\sL{\mathsf L}
\nc\sC{\mathsf C}

\nc\Cone{\mathit Cone}

\nc\fY{\mathfrak Y}
\nc\fe{\mathfrak e}
\nc\ft{\mathfrak t}

\nc\wt{\widetilde}
\nc\inj{\mathit{inj}}
\nc\surj{\mathit{surj}}

\nc\Path{\mathit{Path}}
\nc\Set{\mathit{Set}}
\nc\Fin{\mathit{Fin}}

\nc\cyc{\mathit{cyc}}

\nc\per{\mathit{per}}

\nc\sym{\mathit{symp}}
\nc\con{\mathit{cont}}
\nc\gen{\mathit{gen}}
\nc\str{\mathit{str}}
\nc\rsdl{\mathit{res}}
\nc\impr{\mathit{impr}}
\nc\rel{\mathit{rel}}
\nc\pt{\mathit{pt}}
\nc\naive{\mathit{nv}}
\nc\forget{\mathit{For}}

\nc\sH{\mathsf H}
\nc\sW{\mathsf W}
\nc\sE{\mathsf E}
\nc\sP{\mathsf P}
\nc\sB{\mathsf B}
\nc\sS{\mathsf S}
\nc\fH{\mathfrak H}
\nc\fP{\mathfrak P}
\nc\fW{\mathfrak W}
\nc\fE{\mathfrak E}
\nc\sx{\mathsf x}
\nc\sy{\mathsf y}

\nc\ord{\mathit{ord}}

\nc\sm{\mathit{sm}}

\nc\rhu{\rightharpoonup}
\nc\dirT{\mathcal T}
\nc\dirF{\mathcal F}
\nc\link{\mathit{link}}
\nc\cT{\mathcal T}

\newcommand{\ssupp}{\mathit{ss}}
\newcommand{\cyl}{\mathit{Cyl}}
\newcommand{\ball}{\mathit{B(x)}}

 \nc\ssf{\mathsf f}
 \nc\ssg{\mathsf g}
\nc\sq{\mathsf q}
 \nc\sQ{\mathsf Q}
 \nc\sR{\mathsf R}

\nc\fa{\mathfrak a}
\nc\fA{\mathfrak A}

\nc\trunc{\mathit{tr}}
\nc\pre{\mathit{pre}}
\nc\expand{\mathit{exp}}

\nc\Sol{\mathit{Sol}}
\nc\direct{\mathit{dir}}

\nc\out{\mathit{out}}
\nc\Morse{\mathit{Morse}}
\nc\arb{\mathit{arb}}
\nc\prearb{\mathit{pre}}

\nc\BZ{\mathbb Z}
\nc\proper{\mathit{prop}}
\nc\torsion{\mathit{tors}}


\title[A combinatorial calculation of the Landau-Ginzburg model $M= \BC^3, W=z_1 z_2 z_3$]{A combinatorial calculation of the\\ Landau-Ginzburg model $M= \BC^3, W=z_1 z_2 z_3$}

\author{David Nadler}
\address{Department of Mathematics\\University of California, Berkeley\\Berkeley, CA  94720-3840}
\email{nadler@math.berkeley.edu}

\begin{abstract}
The aim of this paper is to apply ideas from the study of Legendrian singularities to  a specific example of interest within mirror symmetry.
We calculate  the Landau-Ginzburg $A$-model
with $M= \BC^3, W=z_1 z_2 z_3$ in its guise as microlocal sheaves along the natural singular Lagrangian thimble  $L = \Cone(T^2)\subset M$. The description we obtain is immediately equivalent
to the $B$-model of the pair-of-pants $\BP^1 \setminus \{0, 1, \infty\}$ as predicted by mirror symmetry. 
\end{abstract}

\maketitle


\tableofcontents


\section{Introduction}
The aim of this paper is to apply ideas from study of Legendrian singularities to a specific example of interest within mirror symmetry.
In the papers~\cite{Ncyc, Narb, Nexp}, we introduced a class of Legendrian singularities of a simple combinatorial nature, called arboreal singularities for their relations to trees, then presented an algorithm  to deform any Legendrian singularity to a nearby Legendrian with  arboreal singularities. Furthermore, we showed that the category of microlocal sheaves on the original Legendrian singularity is  equivalent to that on the nearby Legendrian. 

We will not directly  use  the above theory, but will rather extract and apply one of its key constructions, formalized in Lemma~\ref{arb lemma} below. It enables the calculation of microlocal sheaves on an $n$-dimensional Legendrian singularity
in terms of constructible sheaves on an $n$-dimensional ball. For general Legendrian singularities, an inductive application of this construction leads to the algorithm
presented  in~\cite{Nexp}. But for Legendrian singularities that are simply cones over smooth manifolds, a single application suffices to give appealing results.

In this paper, we will implement this on a single example: the Landau-Ginzburg $A$-model with background $M= \BC^3$
and superpotential $W= z_1z_2z_3$. (Natural generalizations will appear in forthcoming work~\cite{NW}.)
Beyond the specific calculation, our aim is to advocate for the
elementary tools used to describe the symplectic geometry.  
Due to the fact that the critical locus $\{dW = 0\} \subset M$ is not smooth or proper, there has not yet appeared a definitive account of this Landau-Ginzburg $A$-model. 
Nevertheless, it should not be too complicated since it is expected to be mirror to the $B$-model of the pair-of-pants $\BP^1\setminus \{0, 1, \oo\}$. For further discussion, we recommend the beautiful paper~\cite{AAEKO} establishing mirror symmetry in the other direction, between the Landau-Ginzburg $B$-model with $M=\BC^3$, $W= z_1 z_2 z_3$ and the $A$-model of  $\BP^1\setminus \{0, 1, \oo\}$. (There is also work in progress~\cite{AA} pursuing  results parallel to those of this paper but from a more traditional perspective.)

Our starting point will be the following viewpoint on the $A$-model.
The superpotential $W =  z_1z_2z_3$ suggests a natural conic Lagrangian skeleton
 $$
\xymatrix{
L = \{ (z_1, z_2, z_3) \in M \, |\, W(z_1, z_2, z_3)\in \BR_{\geq 0}, |z_1| = |z_2| = |z_3|\} \subset M
}
$$
which can be regarded as a singular thimble over a nearby vanishing  two-torus $T^2 \subset W^{-1}(1)$. We would like to study a model of $A$-branes running along $L$ (as found in the infinitesimal Fukaya-Seidel category~\cite{Seidel}) or transverse to $L$ (as found in the partially wrapped Fukaya category~\cite{AS, Aur}).
For the moment, let us suggest a model for the first version, but see Remark~\ref{rem versions} for the modifications that
lead to a model of  the second.

\begin{ansatz}
The Landau-Ginzburg $A$-model of $M= \BC^3$ with superpotential $W= z_1z_2z_3$ is given by the dg category of microlocal sheaves along the conic Lagrangian skeleton $L\subset M$.
\end{ansatz}

\begin{remark}
The ansatz is compatible with the broad expectation that given  $L\subset M$ a conic Lagrangian skeleton of an exact symplectic manifold, there should be many equivalent approaches to its ``quantum category" of $A$-branes: the Floer-Fukaya-Seidel theory of Lagrangian intersections and pseudo-holomorphic disks~\cite{Seidel} (analysis); the Kashiwara-Schapira~\cite{KS}  theory of microlocal  sheaves~\cite{KS} (topology); and the theory of holonomic modules over deformation quantizations, exemplified by $\cD$-modules~\cite{bernstein}. The close relation between the last two stems from the Riemann-Hilbert correspondence~\cite{K1, K2, M1, M2}, and there are numerous confirmations of their close relation with the first, including proposals
within the context of mirror symmetry~\cite{TZ} to regard $A$-branes as microlocal sheaves as in the ansatz.
\end{remark}

Now we can state the main theorem of this paper. It is a combinatorial calculation of microlocal sheaves along 
the conic Lagrangian skeleton $L\subset M$ leading to a verification of its expected mirror symmetry with the $B$-model of the pair-of-pants $\BP^1\setminus \{0, 1, \oo\}$.

Fix a field $k$ of characteristic zero.

Introduce the base space $X = \BR^3 \times \BR$, its cotangent bundle $T^*X$,
and its spherically projectivized cotangent bundle $S^*X= (T^*X \setminus X)/\BR_{>0}$.

We explain in Section~\ref{sect: app}  how standard constructions identify the conic Lagrangian skeleton $L\subset M$ with a Legendrian subvariety  $\Lambda\subset S^*X$.

Let $\mu\Sh_{\Lambda}(X)$ denote the dg category of microlocal sheaves of $k$-vector spaces supported along 
the Legendrian subvariety $\Lambda \subset S^*X $. See ~\cite{KS} and the discussion of Section~\ref{sect: arb lemma}
 for a precise definition.

Let $\Coh_\torsion(\BP^1 \setminus\{0, 1, \infty\})$ denote the bounded dg category of  finitely-generated torsion complexes
on $\BP^1 \setminus\{0, 1, \infty\}$. It is nothing more than the dg enhancement of  the bounded derived category of finite-dimensional $k$-vector spaces $V$ equipped with an automorphism $m:V\to V$ such that $1$ is not an eigenvalue of $m$.

\begin{thm}[see Theorem~\ref{main thm}]\label{intro main thm}
There is an equivalence 
$$\xymatrix{
\mu\Sh_{\Lambda}(X) \simeq \Coh_\torsion(\BP^1 \setminus\{0, 1, \infty\})
}
$$
\end{thm}

\begin{remark}\label{rem versions}
This comment falls within the ``categorical functional analysis"  underpinning observed dualities   between infinitesimal Fukaya-Seidel categories and partially wrapped Fukaya categories,
and similarly  between properly-supported coherent complexes and perfect complexes (as established in~\cite{BNP}).

We have selected the  formulation of the theorem from among a collection of closely related assertions. It is the version that matches with the definition of microlocal sheaves that includes the traditional local finiteness of constructibility. Alternatively, one can take
the left hand side to consist of arbitrarily large complexes but nevertheless with the prescribed singular support. Our arguments   then provide an equivalence of this with  
a right hand side comprising all quasi-coherent complexes 
$\on{QCoh}(\BP^1 \setminus\{0, 1, \infty\})$. Going further, for example by passing to compact objects, one can arrive
at a version involving  specifically coherent complexes $\on{Coh}(\BP^1 \setminus\{0, 1, \infty\})$. Taking a step back, our constructions offer a simple
combinatorial model of the geometry of the Lagrangian skeleton that could be substituted for any  prior definition of the left hand side.
\end{remark}

\subsection{Acknowledgements}
I thank  D. Auroux and E. Zaslow for suggesting the application studied in this paper. I  also thank them as well as D. Ben-Zvi, M. Kontsevich, J. Lurie, N. Rozenblyum, V. Shende, N. Sheridan, D. Treumann, and H. Williams for their interest, encouragement, and valuable comments.
Finally, I am grateful to the NSF for the support of grant DMS-1502178.


\section{Arborealization lemma}\label{sect: arb lemma}


\subsection{Preliminaries}

Let $X$ be a real analytic manifold, with cotangent bundle $T^*X\to X$, and
spherically projectivized cotangent  bundle $\pi:S^*X = (T^*\setminus X)/\BR_{>0}\to X$.
For convenience, fix a Riemannian metric on $X$, so that in particular $S^*X$ is identified with the unit cosphere bundle $U^*X \subset T^*X$.

Let $\Lambda\subset S^*X$ be a closed Legendrian subvariety with image $H = \pi(\Lambda)\subset X$.
We will always work in the generic situation where the projection $\pi|_\Lambda:\Lambda\to H$ is finite so that $H\subset X$
is a hypersurface.
For convenience, fix  $\cS=\{X_\alpha\}_{\alpha\in A}$  a Whitney stratification  of $X$ such that $H\subset X$ is a union of strata.  Thus in particular $\Lambda \subset S^*_\cS X = \coprod_{\alpha\in A} S^*_{X_\alpha} X$, where $S^*_{X_\alpha}X\subset S^*X$ denotes the spherically projectivized conormal bundle to the stratum $X_\alpha\subset X$.

Fix a field $k$ of characteristic zero. Let $\Sh(X)$ denote the dg category of  constructible complexes of sheaves of $k$-vector spaces on $X$.
Let $\Sh_\cS(X) \subset \Sh(X)$ denote the full dg subcategory of   complexes constructible with respect to $\cS$.
 Recall to any $\cF \in\Sh(X) $, we can assign its
singular support $\ssupp(\cF) \subset S^*X$ which is a closed Legendrian subvariety.
Let $\Sh_\Lambda(X) \subset \Sh(X)$ denote the full dg category of  complexes with singular support in $\Lambda$.
The inclusion $\Lambda\subset S^*_\cS X$ implies the full inclusion 
$\Sh_\Lambda(X)\subset \Sh_\cS(X)$.

Let $x\in X$ be a point. 
Let $B_x(\rho), C_x(\rho)\subset X$ be the
open ball and sphere of radius $\rho>0$ centered at $x\in X$.
For small enough $\rho>0$, the sphere $C_x(\rho)\subset X$ will be transverse to the stratification $\cS=\{X_\alpha\}_{\alpha\in A}$ in the sense that it is transverse to each stratum $X_\alpha \subset X$.

Fix  small enough $\rho_2>\rho_1>0$, and for simplicity set $B= B_x(\rho_2)$, $C= C_x(\rho_1)$.
Let $\cS_B =\{X_\alpha \cap B \}_{\alpha}$, $\cS_C =\{X_\alpha \cap C \}_{\alpha}$ denote the respective induced Whitney stratifications. 

Let $\Lambda_B = \Lambda \times_X B \subset S^*B$ denote the induced closed Legendrian subvariety. The inclusion $C\subset X$ induces a correspondence
$$
\xymatrix{
S^*X & \ar[l]_-{p} S^*X \times_X C \ar[r]^-{q} & S^*C
}
$$
Let $\Lambda_C = q(p^{-1}(\Lambda)) \subset S^*C$ denote the induced closed Legendrian subvariety.

Restriction along the inclusion $i_C:C\to X$ provides functors
$$
\xymatrix{
i_C^*:\Sh_\cS(X) \ar[r]^-\sim & \Sh_{\cS_C}(C) 
&
i^*_C:\Sh_\Lambda(X) \ar[r]^-\sim & \Sh_{\Lambda_C}(C) 
}
$$

Suppose in addition the fiber $\Lambda_x  = \Lambda \times_X \{x\}$ is a single codirection.
Let $\mu\Sh_{\Lambda_B}(B)$ denote the dg category of  microlocal sheaves on $B$ supported along $\Lambda_B$.
Within our working context, it admits two  simple equivalent descriptions. (In general, the induced map $\pi_0(\Lambda_x)  \to \pi_0(\Lambda_B)$ is a bijection, and one can treat each connected component of $\Lambda_B$ separately.) 

On the one hand, the natural projection functor is an equivalence
$$
\xymatrix{
\Sh_{\Lambda_B}(B)/\Loc(B)\ar[r]^-\sim  & \mu \Sh_{\Lambda_B}(B)
}
$$
where  $\Loc(B) \subset \Sh(B)$ denotes the full dg subcategory of local systems, or in other words complexes with empty singular support.

On the other hand, let $i_x:\{x\}\to X$ denote the inclusion. 
Let $\Sh(B)^0_! \subset \Sh(X)$ denote the full dg subcategory of $\cF \in \Sh(X) $
such that $i^!_x\cF \simeq 0$.
Let $\Sh_{\Lambda_B}(B)^0_! \subset  \Sh_{\Lambda_B}(B) $ denote the full dg subcategory of $\cF \in \Sh_{\Lambda_B}(B) $
such that $i^!_x\cF \simeq 0$.
Alternatively, let $\Gamma_c:\Sh(B)\to \Mod_k$ denote the functor of global sections with compact support. Under the natural identification $\Sh(\{x\}) \simeq \Mod_k$, for any $\cF\in \Sh_{\cS_B}(B)$, in particular for $\cF\in \Sh_{\Lambda_B}(B)$, there is a natural equivalence $i_x^!\cF \simeq \Gamma_c(B, \cF)$, and hence 
$\Sh_{\Lambda_B}(B)^0_! \subset  \Sh_{\Lambda_B}(B) $ is also 
the full dg subcategory of $\cF \in \Sh_{\Lambda_B}(B) $
such that $\Gamma_c(B, \cF) \simeq 0$.
Finally, the natural projection functor is an equivalence
$$
\xymatrix{
\Sh_{\Lambda_B}(B)^0_!\ar[r]^-\sim  & \mu \Sh_{\Lambda_B}(B)
}
$$

\begin{example}\label{ex smooth}
Let $H\subset X$ be a smooth hypersurface given as the zero-locus  of a submersion $f:X\to \BR$.
Let $\Lambda\subset S^*X$ be the smooth Legendrian subvariety  given by the codirection of the differential $df$ along $H$. 
Thus $\Lambda \subset S^*X $ is one of the two connected components of $S^*_H X \subset S^*X$, and the restriction $\pi|_\Lambda:\Lambda\to H$ is a diffeomorphism. 

Let  $x\in H \subset X$ be a point, and $B\subset X$ a small ball around $x$. Let $j:B_+\to B$ be the inclusion of the open subset where $f>0$. 
Then $\Sh_{\Lambda_B}(B)$ is generated by the constant sheaf $k_B$ and  
the standard extension $j_{*}k_{B_+}$. The only nontrivial morphism between them is the canonical map $k_B\to j_*j^*k_B\simeq j_{*}k_{B_+}$. Thus $\Sh_{\Lambda_B}(B)$ is
 equivalent to 
  the dg category of finitely-generated complexes of $k$-modules over the directed $A_2$-quiver
$$
\xymatrix{
\bullet\ar[r] & \bullet
}
$$ 
Furthermore, $\Sh_{\Lambda_B}(B)^0_! $  is generated by the standard extension $j_{*}k_{B_+}$. 
Thus $\Sh_{\Lambda_B}(B)^0_!$ is
 equivalent to  the dg category of finitely-generated complexes of $k$-modules. Note that for $\cF\in \Sh_{\Lambda_B}(B)^0_!$,
 the equivalence is realized by taking the stalk of $\cF$ at any point of $B_+$.
\end{example}

\begin{example}\label{ex transverse}
For $a=1, 2$,
let $H_a\subset X$ be a smooth hypersurface given by the zero-locus  of a submersion $f_a:X\to \BR$.
Let $\Lambda_a\subset S^*X$ be the smooth Legendrian subvariety  given by the codirection of the differential $df_a$ along $H_a$. 
Thus $\Lambda_a \subset S^*X $ is one of the two connected components of $S^*_{H_a} X \subset S^*X$, and the restriction $\pi|_{\Lambda_a}:\Lambda_a\to H_a$ is a diffeomorphism. 

Suppose in addition that $H_1$ and $H_2$ are transverse.
Let  $x\in H_1\cap H_2  \subset X$ be a point, and $B\subset X$ a small ball around $x$. 
For $a=1, 2$, let $j_i:B_{a+}\to B$ be the inclusion of the open subset where $f_a>0$. 
Then $\Sh_{\Lambda_B}(B)$ is generated by the constant sheaf $k_B$ and  
the standard extensions $j_{a*}k_{B_{a+}},$, for $a=1, 2$.
The only nontrivial morphisms between them are the canonical maps $k_B\to j_{a*}j_a^*k_B\simeq j_{a*}k_{B_{a+}}$, 
for $a=1,2$. Thus $\Sh_{\Lambda_B}(B)$ is
 equivalent to the dg category of finitely generated complexes of $k$-modules over the directed $A_3$-quiver
$$
\xymatrix{
\bullet & \ar[l] 
\bullet\ar[r] & \bullet
}
$$ 
Furthermore, $\Sh_{\Lambda_B}(B)^0_! $  is generated by
the standard extensions $j_{i*}k_{B_{a+}},$ for $a=1, 2$.
Thus $\Sh_{\Lambda_B}(B)^0_!$ is
 equivalent to the direct sum of two copies of  the dg category of finitely generated complexes of $k$-modules.
Note that for $\cF\in \Sh_{\Lambda_B}(B)^0_!$,
 the equivalence is realized by taking the stalks $\cF_{b_1}$, $\cF_{b_2}$ at any points $b_1\in B_{1+}\setminus (B_{1+} \cap \ol B_{2+})$,
 $b_2\in B_{2+}\setminus (B_{2+} \cap \ol B_{1+})$ respectively. Moreover, the stalk $\cF_b$ at any point $b\in B_{1+} \cap B_{2+}$ comes equipped
 with a canonical presentation as the direct sum 
 $\cF_b \simeq \cF_{b_1} \oplus \cF_{b_2}$.

\end{example}


\subsection{Statement and proof}

Let us continue with the above setup.

Now we can state a simple-minded but effective construction (at the heart of the arborealization algorithm of~\cite{Nexp}) for calculating $\mu\Sh_{\Lambda_B}(B)$
in its realization as $Sh_{\Lambda_B}(B)^0_!$.

Recall that we have assumed the fiber $\Lambda|_x  = \Lambda \times_X \{x\}$ is a single codirection. Fix a smooth function $f:X\to \BR$ such that $df_x\not = 0$ and represents the codirection $\Lambda|_x$. Then for small enough $\rho_2>\rho_1>0$
as above, within the ball $B= B_x(\rho_2)$, the hypersurface $H =\pi( \Lambda)$ will lie arbitrarily close to the zero-locus of $f$.

Choose a point $c\in C =C_x(\rho_1)$ such that $f(c) < e$ for any $e\in f(H\cap B)$.
Let $A= C\setminus \{c\}$ be the open ball, and note that $\dim A = \dim B -1$.
Write $\Lambda_A\subset S^*A$ for the closed Legendrian subvariety
$\Lambda_C\subset S^*C$ regarded within $S^*A = S^*C \times_C A$.
Let $\Sh_{\Lambda_A}(A)_c \subset \Sh_{\Lambda_A}(A) $ denote the full dg subcategory of complexes with compact support.

Recall that for $\cF\in \Sh_{\Lambda_B}(B)^0_!$, the support of $\cF$ will be disjoint from $c$.

\begin{lemma}[Arborealizarion Lemma]\label{arb lemma}
Restriction induces an equivalence
$$
\xymatrix{
\mu\Sh_{\Lambda_B}(B) \simeq \Sh_{\Lambda_B}(B)^0_! \ar[r]^-\sim & \Sh_{\Lambda_A}(A)_c 
}
$$
\end{lemma}

\begin{proof}
Let $j_U:U = B\setminus \{x\}\to B$ denote the open inclusion. 

For $\cF\in \Sh(B)$, there is a natural triangle
$$
\xymatrix{
i_{x!} i_x^!\cF\ar[r] & \cF\ar[r] & j_{U*} j^*_U \cF\ar[r]^-{[1]} &
}
$$
For $\cF\in \Sh(B)^!_0$, the left term vanishes so there is a natural equivalence  
$$
\xymatrix{
\cF\ar[r]^-\sim & j_{U*} j^*_U\cF
}
$$
or in other words, pushforward induces an equivalence
$$
\xymatrix{
j_{U*}:\Sh_{\Lambda_U}(U) \ar[r]^-\sim & \Sh_{\Lambda_B}(B)_!^0
}
$$
where $\Lambda_U = \Lambda_B \times_B U \subset S^*U$.

Thus the assertion follows from the evident restriction equivalence
$$
\xymatrix{
\Sh_{\Lambda_U}(U) \ar[r]^-\sim & \Sh_{\Lambda_C}(C)
}
$$
and the fact that the support of any $\cF\in \Sh_{\Lambda_B}(B)^0_!$ is disjoint from $c\in C$.
\end{proof}


\subsection{Toy  application}\label{sect: toy app}

Here we will apply Lemma~\ref{arb lemma} in a simple toy situation, but one which will arise subsequently in the proof of Theorem~\ref{main thm}.


Let $D \subset \BR^3$ denote the  two-dimensional open disk 
$$
\xymatrix{
D = \{(\theta_1, \theta_2, \theta_3) \in \BR^3 \, | \, \theta_1^2  + \theta_2^2 + \theta_3^2  < 1, \theta_1 + \theta_2 + \theta_3 = 0 \}
}
$$
Let $W\subset D$ denote the subset where at least one of the coordinates $\theta_1, \theta_2$, or $\theta_3$ is zero.
Note that $W$ is the union of three lines intersecting at the origin.

Introduce the maps given by
$$
\xymatrix{
f:D \ar[r] &  \BR^2 &  f(\theta_1, \theta_2, \theta_3) = (\cos(\theta_1) - \cos(\theta_2), \cos(\theta_2) - \cos(\theta_3) )
}
$$
$$
\xymatrix{
g:D  \ar[r] &  \BR &  g(\theta_1, \theta_2, \theta_3) = \sum_{a = 1}^3 \cos(\theta_a) \sin(\theta_a)
}
$$$$
\xymatrix{
F:D \ar[r] &  \BR^3 & F = (f , g)
}
$$

Here are some of their easily-verified properties: 

\begin{enumerate}
\item $f$ is the quotient map for the $\BZ/2$-action: $(\theta_1, \theta_2, \theta_3) \longmapsto (-\theta_1, -\theta_2, -\theta_3)$; 
\item $g$ is odd: $ g(-\theta_1, -\theta_2, -\theta_3) =  -g(\theta_1, \theta_2, \theta_3)$;
\item $W$ is the zero-locus of $g$, and $g$ is a submersion along $W$.
\end{enumerate}

Property (1) implies $f$, and hence $F$, is an immersion over $D\setminus \{(0, 0, 0)\}$. 
Properties (2) and (3) then imply 
$F$ is an embedding over $D\setminus W$ and self-transverse along $W$.

Set $H = F(D) \subset \BR^3$, and $H^\circ = H \setminus\{(0,0,0)\} \subset \BR^3$. Note that $H^\circ\subset \BR^3$
is immersed with a consistent codirection that is positive on the last coordinate vector field $\partial/\partial{x_3}$.
Let $\Lambda^\circ \subset S^*\BR^3$ denote the corresponding smooth Legendrian submanifold.

 It is straightforward to check that the closure $\Lambda = \ol{\Lambda^\circ}\subset S^*\BR^3$ is a smooth 
 Legendrian submanifold with fiber at the origin the single codirection $dx_3$ corresponding to the last coordinate.

Let $B\subset \BR^3$ be a small ball around the origin, and let $\Lambda_B = \Lambda \times_{\BR^3} B \subset S^*B$
denote the induced closed Legendrian subvariety.

Let us recall the further  ingredients of Lemma~\ref{arb lemma} used to calculate the dg category of microlocal sheaves
$\mu\Sh_{\Lambda_B}(B)$. (In the abstract, this is silly: general theory tells us that since $\Lambda_B$ is smooth, we will find 
$\mu\Sh_{\Lambda_B}(B)$ is equivalent to 
   the dg category of finitely generated complexes of $k$-modules. But Lemma~\ref{arb lemma}  will give an attractive presentation of it that will be compatible with its interaction with other constructions.)

Let $C \subset B$ be a smaller sphere around the origin, and let $c\in C$ be its intersection with the ray $\{(0, 0, -r)\}\subset \BR^3$, for $r>0$.

Let $A = C\setminus \{c\}$ be the open complementary disk, and let $\Lambda_A  \subset S^*A$ denote the induced closed Legendrian subvariety. 

It is straightforward to see that $\Lambda_A$ is diffeomorphic to a circle, 
and projects to the immersed ``trefoil diagram" curve $K \subset A$ as pictured in Figure~\ref{fig: trefoil}. Moreover, one recovers $\Lambda_A$
by taking the ``inward pointing" codirection along $K \subset A$.

\begin{figure}[h]
  \begin{center}  
\includegraphics[scale = .5]{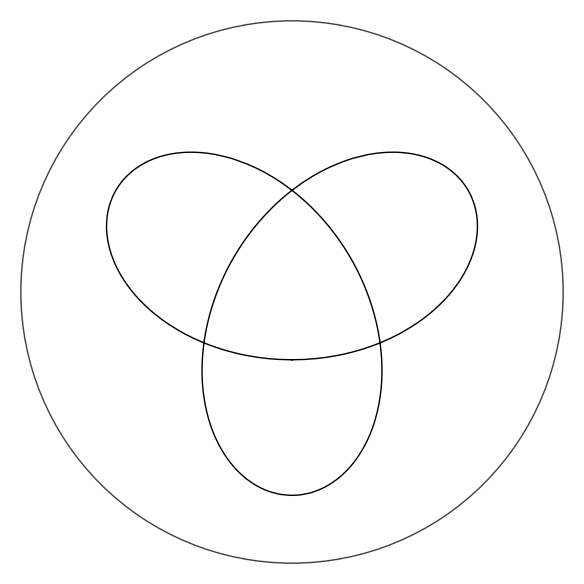}
   \caption{Disk $A$ with immersed ``trefoil diagram" curve $K\subset A$.}
   \label{fig: trefoil}
\end{center}
\end{figure}

Applying Lemma~\ref{arb lemma} provides an equivalence
$$
\xymatrix{
\mu\Sh_{\Lambda_B}(B) \simeq \Sh_{\Lambda_B}(B)^0_! \ar[r]^-\sim & \Sh_{\Lambda_A}(A)_c
}
$$

Now the  calculation of $\Sh_{\Lambda_A}(A)_c$ is appealingly simple.
 As pictured in Figure~\ref{fig: trefoil}, the connected components of $A\setminus K$ consist of a compact contractible central component $U$, three compact contractible neighboring components $U_1, U_2, U_3$, and a non-compact complementary component. Any $\cF\in
\Sh_{\Lambda_A}(A)_c$ vanishes on the non-compact complementary  component,
and  has a respective stalk $V$ and $V_1, V_2, V_3$ along each of the compact contractible components.
Similarly to Examples~\ref{ex smooth} and \ref{ex transverse}, we find that $\Sh_{\Lambda_A}(A)_c$ is thus equivalent
to the bounded  derived category of the category of  four finite-dimensional  $k$-modules $V$ and $V_1, V_2, V_3$, with maps
$V_1\to V$, $V_2\to V$, $V_3\to V$  that induce a direct sum decomposition 
$$
\xymatrix{
 V_a\oplus V_b\ar[r]^-\sim & V,
 &
 \mbox{for any distinct pair of indices $a\not = b$}.
}
$$
 As expected, this is nothing more than the dg category of finitely-generated complexes of $k$-modules, since any such data is equivalent to the data of say $V_1$ alone (starting from the direct sum decomposition $V \simeq V_1 \oplus V_2$, we can view $V_3$ as the graph of an isomorphism $V_1\simeq V_2$).


\section{Landau-Ginzburg model}\label{sect: app}

We  apply here the constructions of Section~\ref{sect: arb lemma}, in particular adopting its notation,
to calculate microlocal sheaves along the natural Lagrangian skeleton of the Landau-Ginzburg model $M=\BC^3$,
$W= z_1z_2 z_3$. 

We  begin by reviewing  some basic constructions in the natural generality of  
the Landau-Ginzburg model $M=\BC^n$,
$W= z_1 \cdots z_n$. For $n>3$, its combinatorics are more involved and we postpone the analogous calculation of
microlocal sheaves along its  natural Lagrangian skeleton to~\cite{NW}.


\subsection{Background}

Let $M=\BC^n$ with coordinates $(z_1, \ldots, z_n)$, where we also write $z_a = x_a + i y_a = r_a e^{i \theta_a}$,
for  $a= 1, \ldots, n$.
Equip $M$ with the exact  symplectic form 
$$
\xymatrix{
\omega_M = \sum_{a=1}^n dx_a dy_a = \sum_{a=1}^n r_a dr_a d\theta_a 
}
$$ with primitive 
$$
\xymatrix{
\alpha_M = \sum_{a=1}^n (x_a dy_a - y_a dx_a) = \sum_{a=1}^n r_a^2 d\theta_a 
}
$$
The associated Liouville vector field $v_M = \sum_{a=1}^n r_a \partial_{r_a}$ generates the positive real dilations.

Let $Y= \BR^n$ with coordinates $(x_1, \ldots, x_n)$, and introduce the exact Lagrangian fibration 
$$
\xymatrix{
p:M\ar[r] &  Y & p(z_1, \ldots, z_n) = (x_1, \ldots, x_n) 
}
$$
There is the evident  section $s:Y\to M$, $s(x_1, \ldots, x_n) = (x_1, \ldots, x_n)$, and 
the function 
$$
\xymatrix{
f:M\ar[r] &  \BR
&
f(z_1, \ldots, z_n) = \sum_{a=1}^n x_a y_a
}
$$ 
is a primitive for the restriction of $\alpha_M$ to each fiber of $p$. 

%

Let $N = M\times \BR = \BC^n \times \BR$ with coordinates $(z_1, \ldots, z_n, t)$.
Equip $N$ with the  contact form 
$$\xymatrix{
\lambda_N = \alpha_M - dt = 
\sum_{a=1}^n (x_a dy_a - y_a dx_a) - dt =  \sum_{a=1}^n r_a^2 d\theta_a - dt
}
$$
and in particular the cooriented contact structure $\xi_N = \ker(\lambda_N) \subset TN$.

Let $X= \BR^n\times \BR$ with coordinates $(x_1, \ldots, x_n, t)$, and introduce the Legendrian fibration 
$$
\xymatrix{
q:N\ar[r] &  X & q(z_1, \ldots, z_n, t) = (x_1, \ldots, x_n, t + \sum_{a=1}^n x_a y_a) 
}
$$

Let $T^*X\to X$ be the cotangent bundle of $X$, and
$\pi:S^*X = (T^*\setminus X)/\BR_{>0}\to X$ its spherically projectivized cotangent  bundle.
The Euclidean metric on $X$ provides an identification of $S^*X$ with the unit cosphere bundle $U^*X \subset T^*X$.

Let us regard $\pi:S^*X \to X$ as a Legendrian fibration and compare it to the 
 Legendrian fibration $q:N\to X$.
 By the Darboux Theorem for Legendrian fibrations, any points $n\in N$ and $\lambda\in S^*X$ admit open neighborhoods $U\subset N$ and $V\subset S^*X$ and a cooriented contactomorphism $\varphi:U\to V$ such that $\varphi(n) = \lambda$ and $\pi\circ \varphi = q$. Thus  the contact geometry of a small ball $U \subset N$ centered at the origin $0\in N$ may be recovered from the differential topology of a small ball $B \subset X$ centered at the origin $0\in X$.


\subsection{Superpotential}

Introduce the superpotential  
$$
\xymatrix{
W:M\ar[r] &  \BC &  W(z_1, \ldots, z_n) = z_1\cdots z_n
}
$$
Its critical locus $\{dW = 0\} \subset M$ is the union of the codimension two coordinate planes 
where two or more  coordinates vanish.
For $n\geq 3$, the critical locus is not proper. 
For $n= 1$, there are no critical values, and for $n\geq 2$, the critical values consist of the single point $0\in \BC$.

The fiber $T  = W^{-1}(1)$ is a complex torus of rank $n-1$ with respect to coordinate-wise multiplication, 
and the restriction $W|_{\BC\setminus\{0\}}  :M|_{\BC\setminus\{0\}} \to {\BC\setminus\{0\}}$ is likewise a $T$-bundle. 
Introduce the maximal compact torus 
$$
\xymatrix{
K = \{(e^{i\theta_1}, \ldots, e^{i\theta_n}) \in M \, |\, \theta_1 + \cdots + \theta_n = 0\} \subset T
}
$$
and its positive real cone
$$
\xymatrix{
L = \{ (re^{i\theta_1} , \ldots, re^{i\theta_n} ) \in M \, |\, \theta_1 + \cdots + \theta_n = 0, \, r\in \BR_{\geq 0} \} \subset M
}
$$

\begin{lemma}
$L \subset M$ is a conic Lagrangian subvariety.
\end{lemma}

\begin{proof}
Clearly $L$ is invariant under  positive real dilations.
Let $v = c_r \partial_r +  c_1 \partial_{\theta_1}+ \cdots +  c_n \partial_{\theta_n}$, with $c_1 + \cdots + c_n=0$, be a tangent vector to the smooth locus $L\setminus \{0\} $. Then $i_v \omega= \sum_{a=1}^n  (c_r r_a  d\theta_a - c_a r_a  dr_a) = 0$,  since along  
$L\setminus \{0\}$, we have $\sum_{a=1}^n  c_r r_a  d\theta_a = c_r r \sum_{a=1}^n    d\theta_a = 0$ and 
$ n\sum_{a=1}^n  c_a r_a  dr_a = (\sum_{a=1}^n  c_a)  r dr = 0$,
where $r= r_1 = \cdots = r_n$.
\end{proof}

Since $L\subset M$ is a conic Lagrangian subvariety,  $ L\times \{t\} \subset M\times \{t\} \subset N$ is a Legendrian subvariety for any choice of $t\in \BR$. Any lift of $L\subset M$ to a Legendrian subvariety $\cL\subset N$ is of this form, and
we will write $\cL  =L\times \{0\}  \subset M\times \{0\} \subset N$ for the lift where $t = 0$.


\subsection{Mirror symmetry}
Now let us focus on the case $n=3$ so that we have the Landau-Ginzburg model $M=\BC^3$, $W= z_1 z_2 z_3$.

Recall the base space $Y = \BR^3$,
and the Lagrangian fibration 
$$
\xymatrix{
p:M\ar[r] &  Y & q(z_1, z_2, z_3) = (x_1, x_2, x_3)
}
$$
given by taking the real parts of vectors. Note that $p$ is equivariant for real dilations.

Let us focus on its restriction to the Lagrangian subvariety 
 $$
\xymatrix{
L = \{ (re^{i\theta_1} , re^{i\theta_2}, re^{i\theta_3}) \in M \, |\, \theta_1 + \theta_2 + \theta_3 = 0, \, r\in \BR_{\geq 0} \} \subset M
}
$$
and in particular its link at the origin
$$
\xymatrix{
K = \{(e^{i\theta_1}, e^{i\theta_2}, e^{i\theta_3}) \in M \, |\,  \theta_1 + \theta_2 + \theta_3 = 0\} \subset M
}
$$

Note that $K$ is a two-torus, and $p|_K$ is the quotient map for the $\BZ/2$-action
$$
\xymatrix{
(e^{i\theta_1} , e^{i\theta_2}, e^{i\theta_3}) \ar@{|->}[r] & 
(e^{-i\theta_1} , e^{-i\theta_2}, e^{-i\theta_3})
}
$$
Thus $p|_K$ is a two-fold cover ramified at the four points
$$
\xymatrix{
R=\{(1, 1, 1), (1, -1, -1),  (-1, 1, -1),   (-1, -1, 1)\}\subset K
}
$$and its image $p(K) \subset X$ is a two-sphere.

Next recall the contact manifold $N= M \times \BR = \BC^3 \times \BR$, the base space $X = Y\times \BR = \BR^3 \times \BR$,
and the Legendrian fibration 
$$
\xymatrix{
q:N\ar[r] &  X & q(z_1, z_2, z_3, t) = (x_1, x_2, x_3, t + \sum_{a=1}^3 x_a y_a) 
}
$$
Note that $q$ is equivariant under simultaneous real dilations of the first three coordinates with real squared-dilations of the last.

 Let us focus on its restriction to the Legendrian subvariety 
 $$
\xymatrix{
\cL = \{ (re^{i\theta_1} , re^{i\theta_2}, re^{i\theta_3}, 0 ) \in N \, |\, \theta_1 + \theta_2 + \theta_3 = 0, \, r\in \BR_{\geq 0} \} \subset N
}
$$
and in particular its link at the origin
$$
\xymatrix{
\cK = \{(e^{i\theta_1}, e^{i\theta_2}, e^{i\theta_3}, 0) \in N \, |\,  \theta_1 + \theta_2 + \theta_3 = 0\} \subset N
}
$$

Note that $\cK$ is a two-torus, and $q|_{\cK}$ is an immersion away from
$$
\xymatrix{
\cR=\{(1, 1, 1, 0), (1, -1, -1, 0),  (-1, 1, -1, 0),   (-1, -1, 1, 0)\}\subset \cK
}
$$
Moreover, note that $q|_{\cK}$ is an embedding away from 
$$
\xymatrix{
\cW=\{(e^{i\theta_1}, e^{i\theta_2}, e^{i\theta_3}, 0)\in \cK \, |\,   \theta_a = 0, \mbox{for some $a=1, 2, 3$} \} \subset \cK
}
$$
and is self-transverse along $\cW$.

Altogether, we see that $H = q(\cL)\subset X$ is a cone (with respect to simultaneous real dilations of the first three coordinates and real squared-dilations of the last) over the surface $S = q(\cK)\subset X$, and so in particular is a hypersurface. 
 
 Now recall that we can identify as contact manifolds a small ball around the origin $0\in N$ with a small ball around the codirection $[dt|_0] \in S^*_0 X$ at the origin $0\in X$ so that the Legendrian fibrations $q:N\to X$, $\pi:S^*X \to X$ agree.

Let us write $\Lambda\subset S^*X$ for the natural closed conic 
(with respect to simultaneous real dilations of the first three coordinates and real squared-dilations of the last)
Legendrian subvariety that agrees with the transport  of $\cL \subset N$ near the origin. To be precise, 
we recover $\Lambda\subset S^*X$ from the hypersurface $H\subset X$ as follows.
Let $H^\sm \subset H$ be any dense open smooth locus, and $S^*_{H^\sm} X \subset S^* X$ its  spherically projectivized conormal bundle.  There is a distinguished codirection $\sigma:H^\sm\to S^*X$ such that $\Lambda|_{H^\sm} = \sigma(H^\sm)$, and the closure of $\sigma(H^\sm)\subset S^*X$
recovers $\Lambda\subset S^*X$.

Now we are in the situation of Section~\ref{sect: arb lemma}, and can apply Lemma~\ref{arb lemma} to calculate
the dg category of microlocal sheaves $\mu\Sh_{\Lambda}(X)$. 
We will henceforth  adopt the constructions and notations introduced therein.
Note that our current situation is slightly simpler: the closed Legendrian subvariety 
$\Lambda\subset S^*X$
is conic 
(with respect to simultaneous real dilations of the first three coordinates and real squared-dilations of the last)
so there is no difference between working over all of $X$ and in a small open ball $B\subset X$ around the origin.


\begin{thm}\label{main thm}
There is an equivalence 
$$\xymatrix{
\mu\Sh_{\Lambda}(X) \simeq \Coh_\torsion(\BP^1 \setminus\{0, 1, \infty\})
}
$$
\end{thm}

\begin{proof} 

Let $C \subset B$ be a three-dimensional sphere around the origin, and let $c\in C$ be its intersection with the ray $\{(0, 0, 0, -r)\}\subset \BR^4$, for $r>0$.

Let $A = C\setminus \{c\}$ be the open complementary three-dimensional ball, and let $\Lambda_A  \subset S^*A$ denote the induced closed Legendrian subvariety. 

Observe that $\Lambda_A$ is diffeomorphic to a two-torus, and projects to a surface $\Sigma = A\cap H$
along the natural map $\pi:S^*A \to A$. 
We recover $\Lambda_A$
by taking the closure of the ``inward pointing" codirection along the smooth locus of $\Sigma$.

Observe that the key properties of $\pi|_{\Lambda_A}$, and hence of $\Sigma\subset A$, coincide with those of $q|_\cK$ described above. 
Namely, identifying $\Lambda_A$ with the two-torus $\{(e^{i\theta_1}, e^{i\theta_2}, e^{i\theta_3})  \, |\,  \theta_1 + \theta_2 + \theta_3 = 0\}$,
the projection $\pi|_{\Lambda_A}$ is an immersion away from the four points
$$
\xymatrix{
\{(1,1,1), (1, -1, -1),  (-1, 1, -1),   (-1, -1, 1)\}
}
$$
an embedding away from the locus
$$
\xymatrix{
\{(e^{i\theta_1}, e^{i\theta_2}, e^{i\theta_3}) \, |\,  \theta_1 + \theta_2 + \theta_3 = 0, \, \theta_a = 0, \mbox{for some $a=1, 2, 3$} \} 
}
$$
and self-transverse where immersed.
Furthermore, the local geometry of $\Sigma\subset A$ near the four ramified points has been described in 
the toy application of Section~\ref{sect: toy app} above.

Now applying Lemma~\ref{arb lemma} provides an equivalence
$$
\xymatrix{
\mu\Sh_{\Lambda}(X) \simeq \Sh_{\Lambda}(X)^0_! \ar[r]^-\sim & \Sh_{\Lambda_A}(A)_c
}
$$

Following Examples~\ref{ex smooth} and~\ref{ex transverse} and the calculation of Section~\ref{sect: toy app}, we find the 
following simple description of $\Sh_{\Lambda_A}(A)_c$.
The connected components of $A\setminus \Sigma$ consist of a compact contractible central component $U$, four compact contractible neighboring components $U_1, U_2, U_3, U_4$, and a non-compact complementary component. Any $\cF\in
\Sh_{\Lambda_A}(A)_c$ vanishes on the non-compact complementary  component,
and  has a respective stalk $V$ and $V_1, V_2, V_3, V_4$ along each of the compact contractible components.
Similarly to Examples~\ref{ex smooth} and \ref{ex transverse} and
the calculation of Section~\ref{sect: toy app}, we find that $\Sh_{\Lambda_A}(A)_c$ is equivalent
to the  bounded derived category of the category of  five finite-dimensional  $k$-modules $V$ and $V_1, V_2, V_3, V_4$, with maps
$V_1\to V$, $V_2\to V$, $V_3\to V$, $V_4\to V$  that induce a direct sum decomposition 
$$
\xymatrix{
 V_a\oplus V_b\ar[r]^-\sim & V,
 &
 \mbox{for any distinct pair of indices $a\not = b$}.
}
$$
Starting from the direct sum decomposition $V \simeq V_1 \oplus V_2$, we can view $V_3$ and $V_4$ as the graphs of  isomorphisms 
$$
\xymatrix{
m_3:V_1\ar[r]^-\sim &  V_2 & 
m_4:V_1\ar[r]^-\sim &  V_2
}
$$
Up to equivalence, the data are determined by the automorphism 
$$
\xymatrix{
m = m_4 \circ m_3^{-1}: V_1\ar[r]^-\sim &  V_1
}
$$
and the direct sum decomposition 
$
 V_3\oplus V_4\simeq  V,
$
imposes the only constraint that $1$ is not an eigenvalue of $m$.
\end{proof}





\begin{thebibliography}{99}
\bibitem{AA}
M. Abouzaid and D. Auroux,
Homological mirror symmetry for hypersurfaces in $(\BC^*)^n$,
in preparation.


\bibitem{AAEKO}
M. Abouzaid, D. Auroux, A. I. Efimov, L. Katzarkov and D. Orlov,
Homological mirror symmetry for punctured spheres,
J. Amer. Math. Soc. 26 (2013), 1051--1083.

\bibitem{AS}
M. Abouzaid and P. Seidel, An open string analogue of Viterbo functoriality, Geom. Topol. 14 (2010),
627--718.

\bibitem{Aur}
D. Auroux, 
Fukaya categories and bordered Heegaard-Floer homology,
Proc. International Congress of Mathematicians (Hyderabad, 2010), Vol. II, Hindustan Book Agency, 2010, 917--941.


%

\bibitem{BNP}
D. Ben-Zvi, D. Nadler, and A. Preygel, Integral transforms for coherent sheaves, 
arXiv:1312.7164.

\bibitem{bernstein}
J. Bernstein, Algebraic theory of $\cD$-modules, available online.

%
%
%
%
%
%
%
%
%
%
%
%

\bibitem{K1}
M. Kashiwara, Faisceaux constructibles et syst\`emes holon\^omes d'\'equations aux d\'eriv\'ees partielles lin\'eaires \`a points singuliers r\'eguliers", S\'eminaire Goulaouic-Schwartz, 1979--80, Expos\'e 19, Palaiseau: \'Ecole Polytechnique.

\bibitem{K2}
M. Kashiwara, The Riemann-Hilbert problem for holonomic systems, Publications of the Research Institute for Mathematical Sciences 20 (2): 319--365.

\bibitem{M1}
Z. Mebkhout, Sur le probl\'eme de Hilbert-Riemann, Complex analysis, microlocal calculus and relativistic quantum theory (Les Houches, 1979), Lecture Notes in Physics 126, Springer-Verlag, pp. 90--110.

\bibitem{M2}
Z. Mebkhout, Une autre \'equivalence de cat\'egories, Compositio Mathematica 51 (1): 63--88.



\bibitem{KS}
M. Kashiwara and P. Schapira, 
{\sl Sheaves on manifolds}. Grundlehren der Mathematischen Wissenschaften 292, Springer-Verlag (1994).
%
%
%
%
%
%
%
%

%

\bibitem{Ncyc} D. Nadler, Cyclic symmetries of $A_n$-quiver representations, 
Advances in Math. 269 (2015), 346--363.

\bibitem{Narb} D. Nadler, Arboreal singularities, arXiv:1309.4122.


\bibitem{Nexp} D. Nadler, Non-characteristic expansions of Legendrian singularities, arXiv:1507.01513.

%
\bibitem{NW} D. Nadler
and H. Williams, Combinatorial calculations of microlocal sheaves, in preparation.

%

%
%
%
%
%
%
%
%
%
%
 
\bibitem{Seidel} P. Seidel, {\sl Fukaya Categories and Picard-Lefschetz Theory,}
 Zurich Lectures in Advanced Mathematics. European Mathematical Society (EMS), Z\"urich, 2008.

\bibitem{TZ}
D. Treumann and E. Zaslow, Polytopes and Skeleta,  arXiv:1109.4430.

%

\end{thebibliography}
\end{document}